\documentclass{article}

\usepackage{tikz}
\usepackage{verbatim}     
\usepackage{amsmath,amssymb,amsthm}
\usepackage{color}
\usepackage{graphicx}
\usepackage{subcaption}     
\usepackage{IEEEtrantools}

\newtheorem{theorem}{Theorem}

\newtheorem{corollary}[theorem]{Corollary}

{\theoremstyle{definition}
\newtheorem{definition}[theorem]{Definition}

\newtheorem{remark}[theorem]{Remark}
\newtheorem*{remark*}{Remark}

}

\begin{document}
\par\noindent {\Large\bf
A Schwarz Method for the Magnetotelluric Approximation
    of Maxwell's equations
\par}

{\vspace{4mm}\par\noindent {Fabrizio Donzelli$^1$, Martin J. Gander$^2$, and Ronald D. Haynes$^1$} 
\par\vspace{2mm}\par}
{\vspace{2mm}\par\noindent {\it $^{1}$~Department of Mathematics and Statistics, Memorial University of Newfoundland, St. John's (NL) A1C 5S7, Canada}}
{\par\noindent {\it $^{2}$~Section de Math\'ematiques 2-4 rue du Li\`evre, CP 64 CH-1211 Gen\`eve}}
{\vspace{2mm}\par\noindent {\it\textup{E-mail:} fdonzelli@mun.ca, Martin.Gander@unige.ch, rhaynes@mun.ca
}\par}

\abstract{The magnetotelluric approximation of the Maxwell's
  equations is used to model the propagation of low
  frequency electro-magnetic waves in the Earth's subsurface,
    with the purpose of reconstructing the presence of mineral or oil
  deposits. We propose a classical Schwarz method for solving
    this magnetotelluric approximation of the Maxwell equations,
  and prove its convergence using maximum principle
    techniques. This is not trivial, since solutions are complex
    valued, and we need a new result that the magnetotelluric 
    approximations satisfy a maximum modulus principle for our
    proof. We illustrate our analysis with numerical experiments. }

\section{Introduction}\label{sec:1}

Maxwell's equations can be used to model the propagation
of electro-magnetic waves in the subsurface of the
 Earth. The interaction of such waves with the material in the
subsurface produces response waves, which carry information about the
physical properties of the Earth's subsurface, and their measurement
  allows geophysicists to detect the presence of mineral or oil
deposits.  Since such deposits are often found to be invariant with
respect to one direction parallel to the Earth's surface, the model
can be reduced to a two dimensional complex partial differential
equation. Following \cite{Weiss2012}, the magnetotelluric 
  approximation is derived from the full 3D Maxwell's
equations,
\begin{equation}\label{Maxwell3d}
  \frac{\partial B}{\partial t}+\nabla \times E = 0, 
  \quad -\frac{\partial D}{\partial t}+\nabla \times H = J,
\end{equation}
in the quasi-static (i.e.\ long wavelength, low frequency) 
regime, which implies that $\frac{\partial D}{\partial t}$ in
  \eqref{Maxwell3d} is neglected.  Assuming a time dependence of
the form $e^{i\omega t}$, where $\omega$ is the pulsation of
  the wave, using Ohm's law, $J=\sigma E+J^e$, where $J^e$
  denotes some exterior current source, and the constitutive relation
$B=\mu H$ where $\mu$ is assumed to be the permeability of free space,
we obtain
\begin{equation}\label{eq:EHeqn}
  \nabla \times E = -i\omega\mu H,\quad
  \nabla \times H = \sigma E+J^e.
\end{equation}
Assuming the plane-wave source of magnetotellurics, and a
two-dimensional Earth structure such that $\sigma=\sigma(x,z)$, the
electric and magnetic fields can be decomposed into two
independent modes.  For the TM-, or H-polarization, mode, we have
$$ 
E = (E_x,0,E_z) \quad\text{and}\quad H = (0,H_y,0).
$$
Hence the first vector valued equation in (\ref{eq:EHeqn})
becomes a scalar equation,
\begin{equation}\label{eq:Enew}
  \frac{\partial E_x}{\partial z}-\frac{\partial E_z}{\partial x}
  =-i\omega\mu H_y,
\end{equation}
and the second vector valued equation in (\ref{eq:EHeqn})
gives two scalar equations,
\begin{equation}\label{eq:Hnew1}
-\frac{\partial H_y}{\partial z} = \sigma E_x+J^e_x \quad\text{or}
\quad E_x = -\frac{1}{\sigma}\frac{\partial H_y}{\partial z}
-\frac{J_x^e}{\sigma},
\end{equation}
and
\begin{equation}\label{eq:Hnew2}
\frac{\partial H_y}{\partial x} = \sigma E_z+J_z^e \quad\text{or}
\quad E_z = \frac{1}{\sigma}\frac{\partial H_y}{\partial x}
-\frac{J_z^e}{\sigma}.
\end{equation}
Substituting (\ref{eq:Hnew1}) and (\ref{eq:Hnew2}) into
(\ref{eq:Enew}) thus leads to a scalar equation for $H_y$,
\begin{equation}\label{eq:ourmodel}
-\frac{\partial}{\partial z}\left\{\frac{1}{\sigma}\frac{\partial H_y}{\partial z}\right\}
-\frac{\partial}{\partial x}\left\{\frac{1}{\sigma}\frac{\partial H_y}{\partial x}\right\}
+i\omega\mu H_y =\frac{\partial}{\partial z}\left\{
\frac{J_x^e}{\sigma}\right\}
  -\frac{\partial}{\partial x}\left\{\frac{J_z^e}{\sigma}\right\}.
\end{equation}
In geophysical applications the coefficient of conductivity $\sigma$
is in general a non-constant, piece-wise continuous function.  We
  will assume however for simplicity of our first study here that
  $\sigma \equiv 1$. If we then set $u := H_y$, assume homogeneous
    Dirichlet boundary conditions and denote by
    $f:=-\frac{\partial}{\partial z}\left\{
    \frac{J_x^e}{\sigma}\right\} +\frac{\partial}{\partial
      x}\left\{\frac{J_z^e}{\sigma}\right\}$, we obtain as the
  magnetotelluric approximation of the Maxwell equations (cf.\ 
  equation (2.86) in \cite{Weiss2012})
\begin{equation}\label{eq:1}
  \Delta u -i\omega u= f\quad \mbox{in $\Omega$}\;, u=0
  \quad \mbox{on $\partial\Omega$\;.}
\end{equation}
We further assume for simplicity that $\Omega$ is a domain with
smooth boundary and that $f\in C^{\infty}(\Omega) \cap
C(\overline\Omega)$. The pulsation $\omega$ is assumed to be real
and non-zero. Note that the solution $u$ of equation \eqref{eq:1}
  could also represent a component of the electric field if the model
  had been derived in an analogous fashion from the TE mode.

We are interested in solving the magnetotelluric approximation
  \eqref{eq:1} using Schwarz methods.  The alternating Schwarz
method, introduced by H.A. Schwarz in 1869 \cite{Schwarz1869} to
  prove existence and uniqueness of solutions to Laplace's equation on
  irregular domains, is the foundational idea of the field of domain
decomposition, and has inspired work in both theoretical aspects and
applications to all fields of science and engineering, see
  \cite{Gander2008,GanderWanner2014} and references therein for more
  information about the historical context. Lions
\cite{Lions1988,Lions1989} reconsidered the problem of the convergence
of the method for the Poisson equation on more general configurations
of overlapping subdomains. In his second paper \cite{Lions1989}, he
followed the idea of Schwarz and proved convergence of the
alternating Schwarz method using the maximum principle for
harmonic functions. He also introduced a parallel variant of
  the Schwarz method, where all subdomain problems are solved
  simultaneously. Schwarz methods have also been introduced and
  studied for the original Maxwell equations \eqref{Maxwell3d}, see
  \cite{DolGanGer-Gio2009,BouDolGanLan2012,DoleanGanderLee2015,
    BouajajiDoleanGander2015, DoleanGanderVeneros2016,
    DoleanGanderVeneros2018} and references therein; maximum
  principle arguments can not be used here for studying
  convergence.  We show here that for the magnetotelluric  
  approximation of Maxwell equations in \eqref{eq:1}, which also has
  complex solutions like the original Maxwell equations, the
  convergence of the parallel Schwarz method can be proved using a
  maximum modulus principle satisfied by complex solutions of
  \eqref{eq:1}.

\section{Well-Posedness, Schwarz Method and Convergence}\label{sec:2}

We start by establishing the well-posedness of the
  magnetotelluric approximation of the Maxwell equations in
  \eqref{eq:1}.
\begin{theorem}\label{th:1}
  Let $\Omega\subset\mathbb R^2$ be a bounded domain with smooth
  boundary. Assume that $f\in L^2(\Omega)$ and $\omega$ is a
  non-zero constant. Then the 
  boundary value problem \eqref{eq:1} has a unique solution $u \in
  H^1_0(\Omega)$, depending continuously on $f$.
\end{theorem}
\begin{proof}
  This result follows from a standard application of the Riesz
  Representation Theorem and the Lax-Milgram Lemma.
\end{proof}

We now decompose the domain $\Omega\subset\mathbb R^2$ first into
  non-overlapping subdomains, and then enlarge each subdomain by a
  layer of positive width to obtain the overlapping subdomains
$\Omega_j$, for $j = 1, ... J$, leading to a strongly overlapping
  subdomain decomposition of $\Omega$.  An example is shown in Figure
\ref{fig:1}, where the non-overlapping decomposition is indicated
  by the dashed lines, see also \cite{GanderZhao2000}.
\begin{figure}[t]
\centering
\begin{tikzpicture}
\draw (1,0) -- (5,0);
\draw (5,0) arc (270:360:1cm);
\draw (1,0) arc (270:180:1cm);
\draw[dashed] (3,0) -- (3,6);
\draw[dashed] (0,3) -- (6,3);
\draw (0,1) -- (0,5);
\draw (0,5) arc (180:90:1cm);
\draw (6,1) -- (6,5);
\draw (6,5) arc (0:90:1cm);
\draw (1,6) -- (5,6);
%
\draw (0,2.7) -- (3,2.7); 
\draw (3,2.7) arc (270:360:0.3cm); 
\draw (3.3,3) -- (3.3,6); 
\draw (0.5,5.6) node[anchor=west, scale=0.5]{\huge $\Omega_1$}; 
\draw (0,2.4) node[anchor=west, scale=0.5]{\huge $\Gamma_1$}; 
%
\draw (2.7,6) -- (2.7,2.9); 
\draw (2.7,2.9) arc (180:270:0.3cm); 
\draw (3,2.6) -- (6,2.6); 
\draw (5.5,5.6) node[anchor=east, scale=0.5]{\huge $\Omega_2$}; 
\draw (6,2.3) node[anchor=east, scale=0.5]{\huge $\Gamma_2$}; 
%
\draw (2.5,0) -- (2.5,3.2); 
\draw (2.5,3.2) arc (-180:-270:0.3cm); 
\draw (2.8,3.5) -- (6,3.5); 
\draw (5.5,0.4) node[anchor=east, scale=0.5]{\huge $\Omega_3$}; 
\draw (6,3.8) node[anchor=east, scale=0.5]{\huge $\Gamma_3$}; 
%
\draw (0,3.2) -- (3.6,3.2); 
\draw (3.6,3.2) arc (-270:-360:0.3cm); 
\draw (3.9,2.9) -- (3.9,0); 
\draw (0.5,0.4) node[anchor=west, scale=0.5]{\huge $\Omega_4$}; 
\draw (0,3.5) node[anchor=west, scale=0.5]{\huge $\Gamma_4$}; 
\end{tikzpicture}
\caption{Strongly overlapping subdomain decomposition obtained by
    enlarging a non-overlapping decomposition, indicated by the dashed
    lines, by a layer of strictly positive width}
\label{fig:1}
\end{figure}
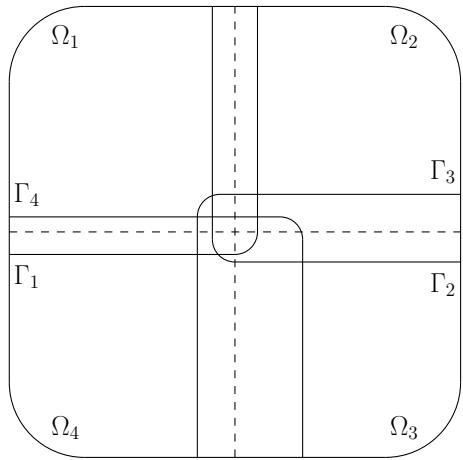
For such strongly overlapping decompositions, one can define a
  smooth partition of unity $\{\chi_j\}_{j=1}^J$ subordinated to the open
  covering $\{\Omega_j\}_{j=1}^J$, such that the support of $\chi_j$
  is a set $K_j$ contained in the open subdomain $\Omega_j$ for each
  $j=1,2,\ldots,J$, see \cite[Theorem 15, Chapter 2]{Spivak1999}.
The assumption of a strongly overlapping decomposition is not strictly
necessary to use maximum principle arguments, see for example
  \cite{CiaramellaGander2018,CiaramellaGander2018DD25}, which contain
  even accurate convergence estimates, but we make it here since
it simplifies the new and first application of the maximum
modulus principle (via Corollary \ref{cor:1}) for studying Schwarz
  methods for equations with complex valued solutions. For each
$\Omega_j$, we denote by $\Gamma_j$ the portion of
$\partial\Omega_j$ in the interior of $\Omega$.

The parallel Schwarz method for such a multi-subdomain
decomposition starts with a global initial guess for the solution
  of \eqref{eq:1}, $u_{\text{glob}}^0\in C^2(\Omega)\cap
  C^0(\overline \Omega)$ (less regularity would also be
  possible, because of the regularization provided by the
  equation). If at the step $n$ of the parallel Schwarz method
the global approximation $u_{\text{glob}}^n$ has been
constructed, and $u_{\text{glob}}^n\in C^2(\Omega)\cap C^0(\overline
\Omega)$, then it constructs the next global
approximation at step $n+1$ by solving, for $j=1,\cdots ,J$, the
Dirichlet problems
\begin{equation}\label{eq:2}
 \begin{array}{rcll}
  \Delta u^{n+1}_j -i\omega u^{n+1}_j&=& f\quad&\mbox{in $\Omega_j$},\\
  u^{n+1}_j&=&0 &\mbox{on $\overline\Omega_{j}\cap \partial\Omega$},\\
  u^{n+1}_j &=& u^{n}_{\text{glob}}\qquad& \mbox{on $\Gamma_{j}$},
 \end{array}
\end{equation}
and then defining the $(n+1)^{th}$ global iterate by using the partition of unity,
\begin{equation}\label{eq:3}
  u_{\text{glob}}^{n+1}=\sum_{j=1}^J\chi_ju^{n+1}_j\;. 
\end{equation}
Since the initial guess $u_{\text{glob}}^0$ is smooth, by
induction it follows that $u_{\text{glob}}^{n+1}\in C^2(\Omega)\cap
C^0(\overline \Omega)$. This fact allows us to use the classical
(i.e. non-variational ) formulation of the maximum modulus principle.
\begin{definition}
  A real valued function $v$ of class $C^2 (\Omega)$ is said to be
  subharmonic if $\Delta v (x) \geq 0 $, $\forall x \in\Omega$, and
  strictly subharmonic if $\Delta v (x)> 0$, $\forall x \in\Omega$.
\end{definition}
Note that the above definition is not the most general one, but
it is suitable for the purposes of our paper.  The property that we
will use to prove the convergence of the parallel Schwarz method
is the well-known maximum principle, which is the content of the next
theorem (see \cite{GunningRossi1931}, Theorem J-7).
\begin{theorem}\label{th:2}
  Let $v\in C^2(\Omega)\cap C (\overline{\Omega})$ be a non-constant
  subharmonic function.  Let $O\subset \Omega$  be a proper open subset. Then $v$ satisfies the strong maximum
  principle, namely $\max_O v <
  \max_{\partial\Omega} v$.
\end{theorem}
By continuity one can immediately deduce the following corollary,
which contains the key estimate for proving the convergence of the
parallel Schwarz method.
\begin{corollary}\label{cor:1}
  Let $u\in C^2(\Omega)\cap C^0(\overline \Omega)$ be a non-constant
  subharmonic function, and $K$ be a closed subset of $\Omega$. Then
  there exists a constant $\gamma \in [0,1)$ such that $\max_K
    u< \gamma \max_{\partial\Omega} u$.
\end{corollary}
\begin{proof}
  It follows by observing that since $K$ is closed, it has a
  positive distance from the boundary $\partial\Omega$
  of the open domain $\Omega$.
\end{proof}
Since the solution of the magnetotelluric approximation
  \eqref{eq:1} of Maxwell's equation has complex valued solutions, it
  is not directly possible to use the maximum principle result in
  Corollary \ref{cor:1} for proving convergence of the associated
  Schwarz method \eqref{eq:2}-\eqref{eq:3}. The key additional
  ingredient is to prove the following property on the modulus
  of solutions of the magnetotelluric approximation:
\begin{theorem}\label{th:3}
  Let $u\in C^2(\Omega)\cap C (\overline{\Omega})$ be a non-zero
  solution of the homogeneous form of equation \eqref{eq:1}.  Then $|u|^2$ is
  a non-constant subharmonic function.
\end{theorem}
\begin{proof}	
Taking the complex conjugate of the partial differential equation
\eqref{eq:1} with $f=0$, gives a pair of equations,
\begin{equation*}\label{eq:4}
  \Delta u -i\omega u =0\quad\mbox{and}\quad
  \Delta \overline{u} + i\omega\overline u = 0 .
\end{equation*}
Hence we can compute 
\begin{eqnarray*}
  \Delta |u|^2 = \Delta (u\overline u)
  = \nabla (\overline u\nabla u + u \nabla \overline u)
  = \nabla \overline u \nabla u + \overline u \Delta u
  + \nabla u \nabla \overline u + u \Delta \overline u =&\\
  = 2 |\nabla u |^2 +i\omega |u|^2 - i\omega |u|^2
  = 2 |\nabla u |^2 \geq 0\;.&
\end{eqnarray*}
Therefore we have shown that $|u|^2$ is subharmonic.  If $|u|^2$
is constant, the same calculations show that $\nabla u \equiv 0$,
which implies that $u$ is a constant solution, hence it must be
identically equal to zero, since the equation $\Delta u -i\omega u =
0$ has no constant non-zero solutions.
\end{proof}	
We can now prove our main result, namely the convergence of
  the parallel Schwarz method for solving the magnetotelluric 
  approximation of Maxwell's equation in the infinity norm, which we
  denote by $||\cdot||_{S}$ for any function on a set $S$.
\begin{theorem}\label{th:4}
  The parallel Schwarz method \eqref{eq:2}-\eqref{eq:3}
    for the magnetotelluric approximation \eqref{eq:1} of
  Maxwell's equations is convergent and satisfies the error
    estimate
    \begin{equation}\label{eq:ConvEst}
    \max_{j=1,\ldots,J}||u-u_{j}^{n}||_{\Omega_{j}}
      \le \gamma^n\max_{j=1,\ldots,J}||u-u_{j}^{0}||_{\Omega_{j}},
    \end{equation}
    where $u$ denotes the global solution of problem \eqref{eq:1} and
    $u_j^n$ the approximations from the parallel Schwarz method
    \eqref{eq:2}-\eqref{eq:3}, and the constant $\gamma<1$ comes from
    Corollary~\ref{cor:1}.
\end{theorem}
\begin{proof}
  For $j=1\cdots J$, let $K_j\subset\Omega_j$ be the support of the
  partition of unity function $\chi_j$, and let $e_j^n :=u-u_j^n$
 be the error. Then $e_j^n$ is solution of the homogeneous
  equation $\Delta e_j^n - i\omega e_j^n =0$, and hence by
    Theorem \ref{th:3} its modulus is a subharmonic function, and thus
    by Theorem \ref{th:2}, the modulus of the error error $|e^n_j|$
  satisfies the strong maximum principle. We can then estimate on
    each subdomain $\Omega_j$
\begin{eqnarray*}
  &&||e_j^{n+1}||_{\Omega_j} = ||e_j^{n+1}||_{\Gamma_j}
  = ||\sum_{j' = 1}^J \chi_{j'}e_{j'}^{n}||_{\Gamma_j}\\
  && \qquad \le \max_{j'=1,\ldots,J}||e_{j'}^{n}||_{K_{j'}}
  \leq \gamma\max_{j'=1,\ldots,J}||e_{j'}^{n}||_{\Gamma_{j'}}
  =\gamma\max_{j'=1,\ldots,J}||e_{j'}^{n}||_{\Omega_{j'}},
\end{eqnarray*}
where $\gamma \in [0,1)$ is the maximum of the factor introduced in
  Corollary \ref{cor:1} over all $\Omega_j$ and corresponding
  $K_j$. Since this holds for all $j$, we can take the maximum on the
  left and obtain
  $$
  \max_{j=1,\ldots,J}||e_{j}^{n+1}||_{\Omega_{j}}
  \le \gamma\max_{j'=1,\ldots,J}||e_{j'}^{n}||_{\Omega_{j'}},
  $$
  which proves by induction \eqref{eq:ConvEst}.
\end{proof}

\begin{remark}
  The convergence factor $\gamma<1$ is not quantified in Theorem
    \ref{th:4}, since Corollary \ref{cor:1} does not provide a method
  to estimate the constant $\gamma$ in the generality of the
    decomposition we used, but such an estimate is possible for
    specific decompositions, see for example
    \cite{CiaramellaGander2018,CiaramellaGander2018DD25}.
\end{remark}

\begin{remark}
  In \cite{Lions1989}, Lions proved the convergence of the classical
  Schwarz method for the Poisson equation with Dirichlet boundary
  conditions using a method that does not use the maximum
  principle. His remarkable proof is based on the method of orthogonal
  projections, and relies on the fact that the bilinear form
  associated with the weak formulation of the Poisson equation is an
  inner product in the solution space $H^1_0(\Omega)$. We do not see
  how this method can be extended to prove convergence of the
  classical Schwarz method applied to the magnetotelluric  
    approximation of the Maxwell equations. In our case in fact
  the bilinear form associated to the weak formulation of \eqref{eq:1}
  of the global problem is not an inner product, as it fails to be
  symmetric and positive-definite.
\end{remark}

\section{Numerical examples}\label{sec:3}

We now present two numerical experiments. The simulations
are computed on a domain $\Omega$ that consists of two squares
$\Omega_1$ and $\Omega_2$, each of unit size $1\times 1$. The
discretization for each square consists of a uniform grid of $30\times
30$ points. The overlap is along a vertical strip whose width is
specified by the number of grid points, denoted by $d$.

We first compute the error
  $e_j^n:=u-u_j^n$, as used in the proof of Theorem
  \ref{th:4}. In Figure \ref{fig:2}
\begin{figure}[t]
  \centering
  \includegraphics[scale=.20,trim=140 60 100 20,clip]{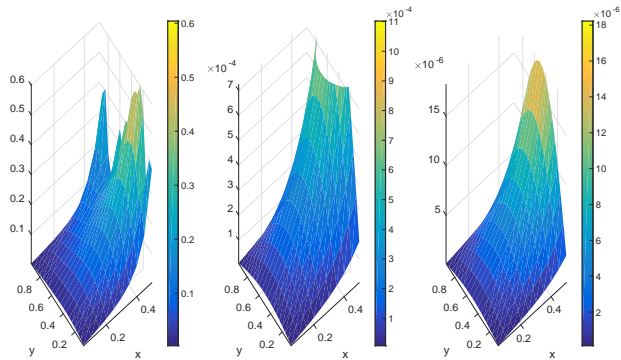}
  \caption{Modulus of the error $e_1^n:=u-u_1^n$ for $n=1,5,15$ on
      the left subdomain when using the parallel Schwarz method for
      solving $\Delta u - i\omega u = 0$. Note how the modulus
    satisfies the maximum principle.}
  \label{fig:2} 
\end{figure}
we show, from left to right, the modulus of the error on the
  left subdomain for iteration $n=1$, $n=5$ and $n=15$, for an overlap
  of $d=6$ horizontal grid points. We chose $\omega=1$, and the
  initial error was produced by generating random values uniformly
  distributed on the range $[0,1]$. Note how the modulus of the error
  clearly satisfies the maximum principle. In Figure \ref{fig:3},
\begin{figure}[t]
  \centering
  \includegraphics[scale=.55]{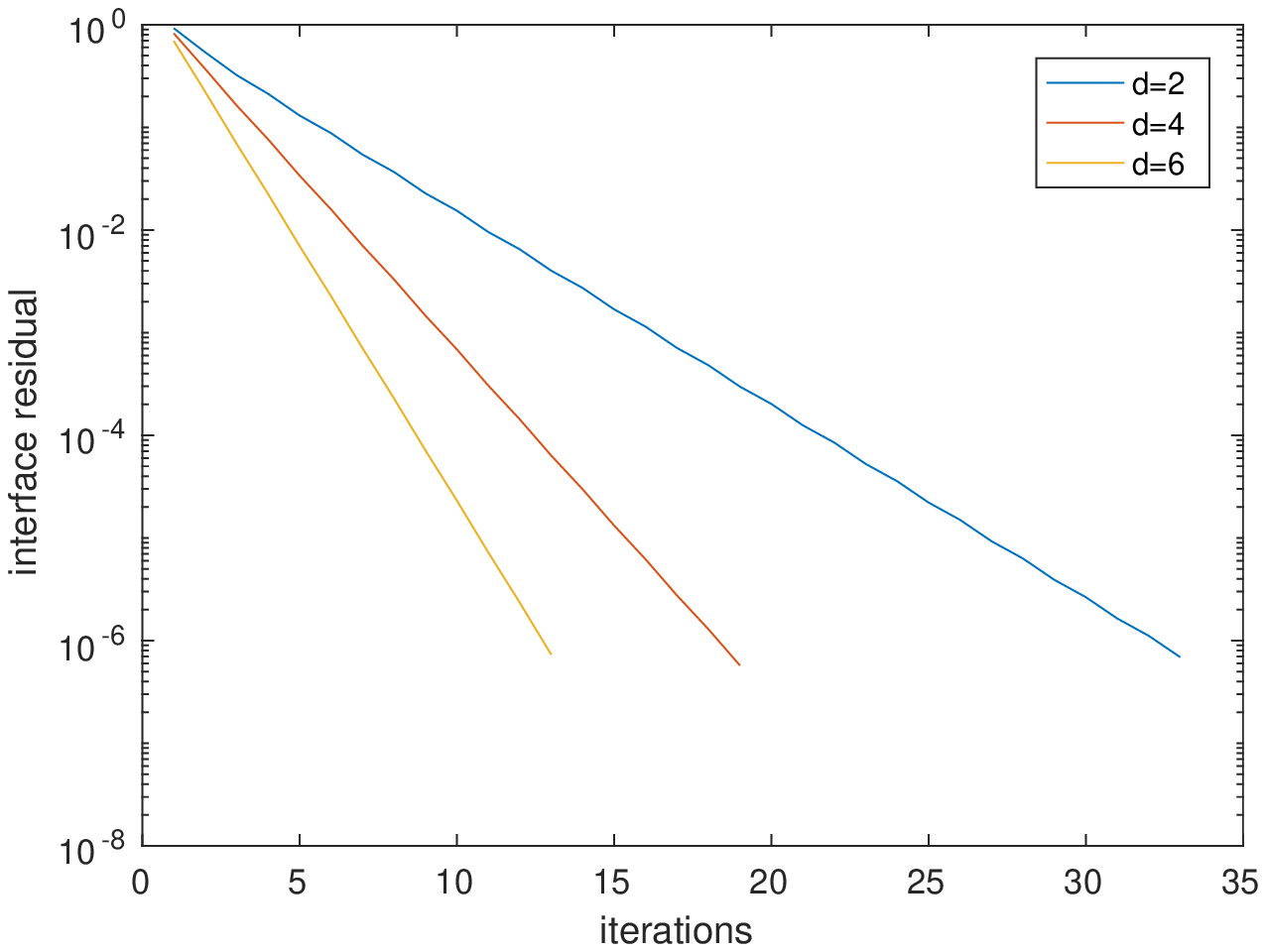}
  \caption{Decay of the interface residual in the 2-norm as a function of the iteration number
      when using the parallel Schwarz method for solving $\Delta u -
      i\omega u = 0$ using different overlap sizes ($d$ denotes the
      number of grid points in the overlap).}
  \label{fig:3} 
\end{figure}
we plot the dependence of the interface residual in the 2-norm on the iteration number, for
three different overlap sizes $d=2,4,6$. As expected, the
performance of the algorithm improves as we increase the size of the
overlap, since increasing the overlap improves the coefficient
  $\gamma$ in Corollary \ref{cor:1} which is the key quantity
  governing the convergence of the parallel Schwarz method.

\section{Conclusion} \label{sec:4}

We showed in this paper that even though the solutions of the
  magnetotelluric approximation of Maxwell's equations are complex
  valued, maximum principle arguments can be used to prove convergence
  of a parallel Schwarz method.  The main new
  ingredient is a maximum modulus principle which is satisfied by the
  solutions of the magnetotelluric approximation. In a forthcoming
  paper, we will analyze the convergence rate of the parallel Schwarz
  method via Fourier analysis, and we will also introduce more
  efficient transmission conditions of Robin (or higher-order) type at
  the interfaces between the subdomains, which leads to optimized
  Schwarz methods, see \cite{Japhet1998,Gander2006} and references
  therein.
  
{\bf Acknowledgement}
We would like to thank Dr.\ Hormoz Jandahari and Dr.\ Colin Farquharson for the insightful explanation of the physics beyond the model discussed in this manuscript. 
The research of the first author was funded, in part, by the Canada Research Chairs program, the NSERC Discovery Grant program and the IgniteR{\&}D program of the Research and Development Corporation of Newfoundland and Labrador (RDC).
\newpage
%
%
%

\end{document}